\documentclass[epsfig,latexsym,amsfonts,twoside]{article}
\usepackage{amssymb,enumerate}
\usepackage{amsmath,amscd}

\usepackage{tikz}
\tikzset{help lines/.style={color=blue!50,very thin}}

\def\part#1{\frac{\partial\phantom{#1}}{\partial#1}}
\newtheorem{thm}{Theorem}

\newtheorem{lemma}[thm]{Lemma}

\newenvironment{proof}{\begin{trivlist}\item[]{\bf Proof} }%
{\hfill $\Box$ \end{trivlist}}
\newenvironment{definition}{\begin{trivlist}\item[]{\bf Definition}\em }%
{\end{trivlist}}
\newenvironment{remark}{\begin{trivlist}\item[]{\bf Remark} }%
{\end{trivlist}}
\newenvironment{example}{\begin{trivlist}\item[]{\bf Example} }%
{\end{trivlist}}
\newenvironment{question}{\begin{trivlist}\item[]{\bf Question} }%
{\end{trivlist}}

\def\Z{\ifmmode{{\mathbb Z}}\else{${\mathbb Z}$}\fi}
\def\Q{\ifmmode{{\mathbb Q}}\else{${\mathbb Q}$}\fi}
\def\C{\ifmmode{{\mathbb C}}\else{${\mathbb C}$}\fi}
\def\P{\ifmmode{{\mathbb P}}\else{${\mathbb P}$}\fi}

\def\H{\ifmmode{{\mathrm H}}\else{${\mathrm H}$}\fi}

\def\B{\ifmmode{{\mathcal B}}\else{${\mathcal B}$}\fi}
\def\E{\ifmmode{{\mathcal E}}\else{${\mathcal E}$}\fi}
\def\F{\ifmmode{{\mathcal F}}\else{${\mathcal F}$}\fi}
\def\K{\ifmmode{{\mathcal K}}\else{${\mathcal K}$}\fi}
\def\L{\ifmmode{{\mathcal L}}\else{${\mathcal L}$}\fi}
\def\M{\ifmmode{{\mathcal M}}\else{${\mathcal M}$}\fi}
\def\N{\ifmmode{{\mathcal N}}\else{${\mathcal N}$}\fi}
\def\O{\ifmmode{{\mathcal O}}\else{${\mathcal O}$}\fi}
\def\U{\ifmmode{{\mathcal U}}\else{${\mathcal U}$}\fi}
\def\V{\ifmmode{{\mathcal V}}\else{${\mathcal V}$}\fi}
\def\X{\ifmmode{{\mathcal X}}\else{${\mathcal X}$}\fi}

\def\Br{\ifmmode{{\mathrm{Br}}}\else{${\mathrm{Br}}$}\fi}
\def\OG{\ifmmode{\widetilde{\cal M}_4}\else{$\widetilde{\cal M}_4$}\fi}
\def\D{\ifmmode{{\mathcal D}^b}\else{${{\mathcal
    D}^b}$}\fi}
\def\Shah{\ifmmode{\amalg\hspace*{-3.5pt}\amalg}\else{$\amalg\hspace*{-3.5pt}\amalg$}\fi}

\begin{document}

\title{Deformations of compact Prym fibrations to Hitchin systems\footnote{2020 {\em Mathematics Subject
  Classification.\/} 14D06, 14J42, 53C26.}}
\author{Justin Sawon and Chen Shen}
\date{February, 2021}
\maketitle

\begin{abstract}
We describe certain Lagrangian fibrations by Prym varieties on holomorphic symplectic varieties that deform to compactifications of $\mathrm{Sp}$-Hitchin systems.
\end{abstract}

\maketitle

\section{Introduction}

The Beauville-Mukai integrable system is given by taking the relative compactified Jacobian of a complete linear system of curves in a K3 surface~\cite{beauville99, mukai84}. The total space is a holomorphic symplectic manifold and the fibres are Lagrangian. The $\mathrm{GL}$-Hitchin integrable system is given by taking the moduli space of $\mathrm{GL}$-Higgs bundles on a Riemann surface $\Sigma$ and mapping to the eigenvalues of the Higgs field~\cite{hitchin87}. The fibres are (compactified) Jacobians of spectral curves, which lie in $T^*\Sigma$. The total space is a non-compact holomorphic symplectic manifold but it admits a natural compactification induced by the one-point compactification of $T^*\Sigma$. Donagi, Ein, and Lazarsfeld~\cite{del97} showed that if $\Sigma$ is contained in a K3 surface, the Beauville-Mukai system of this K3 surface can be deformed to the above compactification of the $\mathrm{GL}$-Hitchin system.

The goal of the present article is to extend this deformation result to the $\mathrm{Sp}$-Hitchin system~\cite{hitchin87}. In this case, one can show that the spectral curves are invariant under multiplication by $-1$ in the fibres of $T^*\Sigma$. Thus the spectral curves are double covers of curves in the total space of $K^2_{\Sigma}$, the quotient of $T^*\Sigma=K_{\Sigma}$ by this involution. The fibres of the $\mathrm{Sp}$-Hitchin system are Prym varieties of these double covers of curves. To construct the corresponding compact Lagrangian fibration by Prym varieties we follow the ideas of Markushevich and Tikhomirov~\cite{mt07} (also Arbarello, Sacc{\`a}, and Ferretti~\cite{asf15} and Matteini~\cite{matteini14, matteini16}). Let $S$ be a K3 surface that is a (branched) double cover of a del Pezzo surface $T$. We take a complete linear system of curves in $T$, their double covers in $S$, and construct the relative Prym variety of this family. The total space is a holomorphic symplectic variety, and the Prym fibres are Lagrangian. We call this the {\em K3-del Pezzo system of $S/T$\/}. 

Now suppose that $\Sigma$ is isomorphic to the ramification locus $\Delta_S\subset S$ of the double cover $S\rightarrow T$. By deforming the K3 surface $S$ to the one-point compactification of $T^*\Sigma=K_{\Sigma}$, and the del Pezzo surface $T$ to the one-point compactification of $K^2_{\Sigma}$, we induce a deformation of the K3-del Pezzo system of $S/T$ to a compactification of the $\mathrm{Sp}$-Hitchin system.

\begin{thm}
There exists a $\P^1$-family $\mathcal{P}_{X/Y}(w)\rightarrow\P^1$ of Lagrangian fibrations such that for $t\neq 0$ we can identify $\mathcal{P}_{X/Y}(w)_t$ with the K3-del Pezzo system, while for $t=0$ we can identify $\mathcal{P}_{X/Y}(w)_0$ with a natural compactification of the $\mathrm{Sp}$-Hitchin system. In this way we get a deformation of the K3-del Pezzo system to a compactification of the $\mathrm{Sp}$-Hitchin system.
\end{thm}

This result was announced in Sawon~\cite{sawon20} and described in Chapter 8 of (Chen) Shen's thesis~\cite{shen20}. A possible application is to the P=W conjecture of de Cataldo, Hausel, and Migliorini~\cite{dhm12}, which relates the perverse and weight filtrations on the rational cohomology of a Lagrangian fibration. This conjecture was originally made for Hitchin systems, and proved for $\mathrm{GL}(2)$-, $\mathrm{SL}(2)$-, and $\mathrm{PGL}(2)$-Hitchin systems on curves of arbitrary genus~\cite{dhm12}. A numerical version of the conjecture was proved for Lagrangian fibrations on smooth compact holomorphic symplectic manifolds by (Junliang) Shen and Yin~\cite{sy19} (for a related result in the compact case, see also Harder, Li, Shen, and Yin~\cite{hlsy19}). Subsequently, de Cataldo, Maulik, and Shen~\cite{dms19} proved the conjecture for $\mathrm{GL}(n)$-Hitchin systems on genus two curves by using an analogue of Donagi, Ein, and Lazarsfeld's deformation coming from the deformation of an abelian surface to the normal cone of an embedded curve. The Debarre integrable system on the generalized Kummer variety~\cite{debarre99} deforms to a compactification of the $\mathrm{SL}$-Hitchin system on a genus two curve; however, de Cataldo, Maulik, and Shen~\cite{dms20} showed that there are obstructions to using this deformation to prove the P=W conjecture in this case. Note that these results are for twisted Hitchin systems, where the moduli spaces are smooth. Felisetti and Mauri~\cite{fm20} considered some untwisted Hitchin systems, namely those that admit symplectic desingularizations, and proved the P=W conjecture for their resolutions (and the PI=WI conjecture for intersection cohomology of the singular spaces). Again, these appear as deformations of certain compact Lagrangian fibrations, including O'Grady's spaces, though these deformations are not explicitly used to prove the conjectures. Our deformation could be used to prove the P=W conjecture (or rather the PI=WI conjecture, since our spaces are singular) for some $\mathrm{Sp}$-Hitchin systems. One would need to first prove the conjecture for the K3-del Pezzo systems, and then investigate the behavior of the perverse and weight filtrations under the deformation.

Here is an outline of our paper. In Section~2 we recall the construction of the $\mathrm{GL}$-Hitchin and Beauville-Mukai systems, and describe the deformation of Donagi, Ein, and Lazarsfeld. In Section~3 we recall the construction of the $\mathrm{Sp}$-Hitchin system and describe the construction of the K3-del Pezzo system. We then construct the deformation leading to the proof of Theorem~1 (which is restated in more detail as Theorem~5). 
%In Section~4 we look at some examples, and compare the degrees of the discriminant loci and the structure of the singular fibres of the $\mathrm{Sp}$-Hitchin system and K3-del Pezzo system.

The authors gratefully acknowledge support from the NSF grant DMS-1555206.

\section{$\mathrm{GL}$-Hitchin and Beauville-Mukai systems}

In this section we recall the deformation result of Donagi, Ein, and Lazarsfeld~\cite{del97}.

\subsection{$\mathrm{GL}$-Hitchin systems}

We start by describing the $\mathrm{GL}$-Hitchin system~\cite{hitchin87}. Let $\Sigma$ be a smooth curve of genus $g\geq 2$. A Higgs bundle $(E,\phi)$ is a pair consisting of a holomorphic vector bundle $E$ on $\Sigma$ and a (holomorphic) Higgs field $\phi\in\Gamma(\Sigma,\mathrm{End}E\otimes K_{\Sigma})$. It is stable (semistable) if
$$\frac{\mathrm{deg}F}{\mathrm{rank}F}<(\leq)\frac{\mathrm{deg}E}{\mathrm{rank}E}$$
for all $\phi$-invariant subbundles $F\subset E$. Denote by $\mathcal{M}_{\mathrm{GL}}(n,d)$ the moduli space of ($S$-equivalence classes of) semistable Higgs bundles of rank $n$ and degree $d$; it is a quasi-projective variety. Moreover, if $n$ and $d$ are coprime then every semistable Higgs bundle is stable and $\mathcal{M}_{\mathrm{GL}}(n,d)$ is smooth.

Denote by $\mathcal{N}(n,d)$ the moduli space of stable bundles on $\Sigma$ of rank $n$ and degree $d$. Its tangent space at $E$ is given by $T_E\mathcal{N}(n,d)=\mathrm{Ext}^1(E,E)$, and any choice of $\phi$ in
$$\mathrm{Ext}^1(E,E)^*\cong\Gamma(\Sigma,\mathrm{End}E\otimes K_{\Sigma})$$
gives a stable Higgs bundle $(E,\phi)$. Thus the cotangent bundle $T^*\mathcal{N}(n,d)$ naturally embeds in $\mathcal{M}_{\mathrm{GL}}(n,d)$ as an open subset. Moreover, the cotangent bundle admits a canonical (holomorphic) symplectic structure given by
$$\sigma((\dot{A}_1,\dot{\phi}_1),(\dot{A}_2,\dot{\phi}_2))=\int_{\Sigma}\mathrm{tr}(\dot{A}_1\dot{\phi}_2-\dot{A}_2\dot{\phi}_1)$$
where $\dot{A}_i\in\mathrm{Ext}^1(E,E)=\mathrm{H}^1(\Sigma,\mathrm{End}E)$ and $\dot{\phi}_i\in\Gamma(\Sigma,\mathrm{End}E\otimes K_{\Sigma})$. This structure extends to $\mathcal{M}_{\mathrm{GL}}(n,d)$, making it a symplectic variety.

The $\mathrm{GL}$-Hitchin system~\cite{hitchin87} is an algebraic completely integrable Hamiltonian system on $\mathcal{M}_{\mathrm{GL}}(n,d)$ given by mapping a Higgs bundle $(E,\phi)$ to
$$(\mathrm{tr}\phi,\mathrm{tr}\phi^2,\ldots,\mathrm{tr}\phi^n)\in A_{\mathrm{GL}}:=\bigoplus_{i=1}^n\Gamma(\Sigma,K_{\Sigma}^i).$$
Locally on $\Sigma$, $\phi$ is an $n\times n$ matrix of one-forms, and the map above takes $(E,\phi)$ to (symmetric polynomials in) the $n$ eigenvalues of $\phi$, giving us $n$ local sections of $K_{\Sigma}$. Globally, a point in $A_{\mathrm{GL}}$ corresponds to an $n$-valued multi-section of $K_{\Sigma}=T^*\Sigma$, i.e., a curve $C\subset T^*\Sigma$ that maps $n$-to-$1$ to $\Sigma$ under the projection $\pi:T^*\Sigma\rightarrow\Sigma$. We call $C$ the spectral curve of $(E,\phi)$. The fibre of the $\mathrm{GL}$-Hitchin system above this point is the Jacobian of $C$, or the compactified Jacobian if $C$ is singular. A pair $(C,L)$ consisting of a spectral curve $C$ and a line bundle $L\in\mathrm{Jac}C$ is known as the spectral data of $(E,\phi)$. Note that $C$ encodes the eigenvalues of $\phi$ and $L$ encodes the eigenspaces. We can recover $(E,\phi)$ from its spectral data $(C,L)$ by $E=\pi_*L$ and $\phi$ is $\pi_*$ of the map
$$L\longrightarrow L\otimes\pi^*K_{\Sigma}$$
given by tensoring with the canonical section of $\pi^*K_{\Sigma}$. This is known as the spectral construction. In order to get $\mathrm{deg}E=d$, the line bundle $L$ on $C$ must have degree $d+n(n-1)(g-1)$. Note that $C$ has genus
$$n^2(g-1)+1=\mathrm{dim}A_{\mathrm{GL}}=\frac{1}{2}\mathrm{dim}\mathcal{M}_{\mathrm{GL}}(n,d).$$

\subsection{Beauville-Mukai systems}

Next we describe the Beauville-Mukai integrable system (first introduced by Mukai~\cite{mukai84}, but later studied in more detail by Beauville~\cite{beauville99}). Let $S$ be a K3 surface containing a smooth curve $\Sigma$ of genus $g\geq 2$. Define $\mathcal{M}_S(v)$ to be the moduli space of semistable sheaves on $S$ with Mukai vector
$$v:=(0,[n\Sigma],k-n^2(g-1))\in\H^{\bullet}(S,\Z).$$
A general element of $\mathcal{M}_S(v)$ looks like $\iota_*L$ where $L$ is a degree $k$ line bundle on a smooth curve $C$ in the linear system $|n\Sigma|$, and $\iota:C\hookrightarrow S$ is the inclusion. Assume that $S$ is generic in the sense that its N{\'e}ron-Severi group $\mathrm{NS}(S)$ is generated over $\Z$ by $\Sigma$. Then if $n$ and $k$ are coprime every semistable sheaf with Mukai vector $v$ is stable and $\mathcal{M}_S(v)$ is smooth and compact. Mukai~\cite{mukai84} proved that $\mathcal{M}_S(v)$ admits a (holomorphic) symplectic structure. There is a natural map $\mathcal{M}_S(v)\rightarrow |n\Sigma|$ taking each sheaf to its (scheme-theoretic) support. Donagi, Ein, and Lazarsfeld~\cite{del97} showed directly that this is a Lagrangian fibration; it also follows by a general result of Matsushita~\cite{matsushita99} that was proved slightly later.

A general point in the base $|n\Sigma|$ represents a smooth curve $C$. The fibre above it is the Jacobian $\mathrm{Jac}^kC$. Thus $\mathcal{M}_S(v)$ can be viewed as a compactification of the relative Jacobian of the family of curves $\mathcal{C}\rightarrow |n\Sigma|$ linearly equivalent to $n\Sigma$. Note that these curves have genus $n^2(g-1)+1$, and this is also the dimension of the base $|n\Sigma|$.

\subsection{Donagi, Ein, and Lazarsfeld's deformation}
\label{K3curve}

Both the $\mathrm{GL}$-Hitchin system and the Beauville-Mukai system are (compactified) relative Jacobians of families of curves in symplectic surfaces, the surfaces being $K_{\Sigma}=T^*\Sigma$ and the K3 surface $S$ respectively. In order to relate the integrable systems, we need to deform one surface to the other. 

First we take the one-point compactification $\overline{K}_{\Sigma}$ of $K_{\Sigma}$. We can think of this as $\P(K_{\Sigma}\oplus\O_{\Sigma})$ with the section at infinity blown down to a single point. If $\Sigma$ is not hyperelliptic then $\overline{K}_{\Sigma}$ can also be identified with the cone over the canonical embedding $\Sigma\hookrightarrow\P(\H^0(\Sigma,K_{\Sigma})^{\vee})\cong\P^{g-1}$. If $\Sigma$ is hyperelliptic, the canonical map is two-to-one onto its image, which is a rational normal curve in $\P^{g-1}$, and $\overline{K}_{\Sigma}$ will be a double cover of the cone over the rational normal curve. For simplicity, we shall assume henceforth that $\Sigma$ is not hyperelliptic.

Now suppose that $\Sigma$ is also contained in a K3 surface $S$, and assume that $S$ is otherwise generic in the sense that $\mathrm{NS}(S)\cong\Z[\Sigma]$. Together with the fact that $\Sigma$ is not hyperelliptic, this implies that $\O(\Sigma)$ is very ample and we have an embedding
$$S\hookrightarrow\P(\H^0(S,\Sigma)^{\vee})\cong\P^g.$$
Restricting to $\Sigma\subset S$ gives the canonical embedding $\Sigma\hookrightarrow\P^{g-1}$ into a hyperplane of $\P^g$. Now take the cone over $S$ in $\P^{g+1}$. There is a pencil of hyperplanes in $\P^{g+1}$ that contain $\Sigma$. Intersecting them with the cone over $S$ gives a pencil of surfaces $X\rightarrow\P^1$ with the following properties:
\begin{itemize}
\item if the hyperplane contains the apex of the cone then the intersection is the cone over the canonical embedding of $\Sigma$, i.e., this surface $X_0$ can be identified with $\overline{K}_{\Sigma}$,
\item if the hyperplane does not contain the apex of the cone then the intersection is isomorphic to $S$ by projecting along the rays of the cone, i.e., these surfaces $X_t$ can all be identified with $S$.
\end{itemize}
Since the pencil of hyperplanes all meet in the $\P^{g-1}$ containing $\Sigma$, $X$ is naturally embedded in $\mathrm{Blow}_{\P^{g-1}}\P^{g+1}$.

\begin{remark}
The family of surfaces $X\rightarrow\P^1$ is actually trivial over $\P^1-\{0\}$. Here is another way to construct it that makes this more readily apparent. Start with the threefold $S\times\P^1$ in $\P^g\times\P^1$ and blow up the curve
$$\Sigma\subset S\times\{0\}\subset S\times\P^1.$$
The normal bundle of $\Sigma$ in $S$ is $K_{\Sigma}$, so the normal bundle of $\Sigma$ in $S\times\P^1$ is $K_{\Sigma}\oplus\O_{\Sigma}$. The exceptional locus is therefore $\P(K_{\Sigma}\oplus\O_{\Sigma})$. Now we blow down the proper transform of $S\times\{0\}$, which meets $\P(K_{\Sigma}\oplus\O_{\Sigma})$ along the section at infinity. This gives precisely $X$. It is a trivial fibration over $\P^1-\{0\}$ with fibres isomorphic to $S$, and the fibre over $0$ is $\P(K_{\Sigma}\oplus\O_{\Sigma})$ with the section at infinity blown down, i.e., it is $\overline{K}_{\Sigma}$. In addition, the corresponding birational modifications of the ambient space $\P^g\times\P^1$, namely, taking the hyperplane
$$\P^{g-1}\subset\P^g\times\{0\}$$
containing $\Sigma$ and blowing it up in $\P^g\times\P^1$, and then blowing down the proper transform of $\P^g\times\{0\}$, produce $\mathrm{Blow}_{\P^{g-1}}\P^{g+1}$, the ambient space of $X$.
\end{remark}

Each of the surfaces $X_t$ contains the curve $\Sigma$, as the zero section of $K_{\Sigma}$ in $X_0\cong\overline{K}_{\Sigma}$ for $t=0$, and through the identification $X_t\cong S$ for $t\neq 0$. Thus they each contain a linear system $|n\Sigma|$. These linear systems each have dimension $n^2(g-1)+1$ and we denote the resulting $\P^{n^2(g-1)+1}$-bundle over $\P^1$ by $P\rightarrow\P^1$. Note that for $t=0$ there is an open subset of $|n\Sigma|$ parametrizing curves that do not meet the apex of $\overline{K_{\Sigma}}$; these are precisely the spectral curves of the $\mathrm{GL}$-Hitchin system, so we can identify this open subset with $A_{\mathrm{GL}}$. 

Let $w\in\mathrm{H}^{\bullet}(X,\Z)$ denote the Chern class of a degree $k$ line bundle on a curve in one of the linear systems $|n\Sigma|$, pushed forward to give a torsion sheaf on $X$. Following Simpson's construction~\cite{simpson94} of moduli spaces of sheaves in the relative setting, we define $\mathcal{M}_{X/\P^1}(w)$ to be the moduli space of semistable sheaves on $X$ with Chern class $w$ that are supported in fibres of $X\rightarrow\P^1$. Each such sheaf will be supported on a curve in a linear system $|n\Sigma|$ for some $X_t$, and thus there is a (scheme-theoretic) support map
$$\begin{array}{ccc}
\mathcal{M}_{X/\P^1}(w) & \longrightarrow & P \\
 & \searrow & \downarrow \\
 & & \mathbb{P}^1.
 \end{array}$$

\begin{thm}[Donagi, Ein, and Lazarsfeld~\cite{del97}]
For $t\neq 0$ we have $X_t\cong S$ and the above map can be identified with the Beauville-Mukai system $\mathcal{M}_S(v)\rightarrow |n\Sigma|$. For $t=0$ we have $X_0\cong\overline{K}_{\Sigma}$ and the above map can be identified with the natural compactification of the $\mathrm{GL}$-Hitchin system $\mathcal{M}_{\mathrm{GL}}(n,d)\rightarrow A_{\mathrm{GL}}$ induced by the inclusions $K_{\Sigma}\subset\overline{K}_{\Sigma}$ and $A_{\mathrm{GL}}\subset |n\Sigma|$. Here $k$ and $d$ are related by $k=d+n(n-1)(g-1)$. In this way we get a deformation of the Beauville-Mukai system to a compactification of the $\mathrm{GL}$-Hitchin system.
\end{thm}

%See~\cite{del97} for additional details.

\begin{remark}
The above construction relies on the property that $\Sigma$ can be embedded in $S$. An interesting question is: when can a given curve be embedded into some K3 surface (in which case we call it a {\em K3 curve\/})? A parameter count shows that this is not possible for a general curve $\Sigma$ of genus $\geq 12$, and thus the $\mathrm{GL}$-Hitchin system associated to $\Sigma$ cannot be deformed to a Beauville-Mukai system. See Arbarello, Bruno, and Sernesi~\cite{abs14, abs17} for recent progress on this question.
\end{remark}

\section{$\mathrm{Sp}$-Hitchin systems and compact Prym fibrations}

Our goal in this section is to describe an analogue for $\mathrm{Sp}$-Hitchin systems of Donagi, Ein, and Lazarsfeld's result.

\subsection{$\mathrm{Sp}$-Hitchin systems}

Previously we described Higgs bundles for the gauge group $\mathrm{GL}(n,\C)$. If instead the gauge group is the symplectic group $\mathrm{Sp}(2n,\C)$ then $E$ will have even rank $2n$ and admit a non-degenerate skew-symmetric two-form, and the Higgs field $\phi$ will lie in $\Gamma(\Sigma,\mathfrak{sp}(2n,\C)\otimes K_{\Sigma})$. Denote by $\mathcal{M}_{\mathrm{Sp}}(2n)$ the moduli space of ($S$-equivalence classes of) semistable $\mathrm{Sp}$-Higgs bundles of rank $2n$; it is a quasi-projective variety. Note that the skew-symmetric two-form on $E$ forces its degree to be zero.

Given an $\mathrm{Sp}$-Higgs bundle $(E,\phi)$, the $2n$ eigenvalues of $\phi$ occur in plus/minus pairs and therefore $\mathrm{tr}\phi^i=0$ for all odd $i$. The Hitchin map is given by mapping an $\mathrm{Sp}$-Higgs bundle $(E,\phi)$ to
$$(\mathrm{tr}\phi^2,\mathrm{tr}\phi^4,\ldots,\mathrm{tr}\phi^{2n})\in A_{\mathrm{Sp}}:=\bigoplus_{j=1}^n\Gamma(\Sigma,K_{\Sigma}^{2j}).$$
This makes $\mathcal{M}_{\mathrm{Sp}}(2n)$ into an algebraic completely integrable Hamiltonian system. Like before, a point in $A_{\mathrm{Sp}}$ corresponds to a $2n$-valued multi-section of $K_{\Sigma}$, i.e., a spectral curve $C\subset K_{\Sigma}$ that projects $2n$-to-$1$ to the zero section $\Sigma$. Because the eigenvalues of $\phi$ occur in plus/minus pairs, the spectral curve $C$ will be invariant under the action of multiplication by $-1$ in the fibres of $K_{\Sigma}$. Quotienting by this action gives a curve $D$ in $K_{\Sigma}/\pm 1\cong K_{\Sigma}^2$, and thus we have a commutative diagram
$$\begin{array}{ccc}
C & \subset & K_{\Sigma} \\
_{2:1}\downarrow & & \downarrow _{2:1} \\
D & \subset & K_{\Sigma}^2. \\
\end{array}$$
Note that $C$ is in $|2n\Sigma|$, the linear system of $2n$ times the zero section in $K_{\Sigma}$, and therefore it has genus $g(C)=4n^2(g-1)+1$. The curve $D$ is in $|n\Sigma|$, the linear system of $n$ times the zero section in $K_{\Sigma}^2$, and therefore it has genus $g(D)=(2n^2-n)(g-1)+1$. Generically $D$ meets the zero section transversely in $4n(g-1)$ points, meaning that the double cover $C\rightarrow D$ has $4n(g-1)$ simple branch points; this agrees with Hurwitz's formula.

Hitchin~\cite{hitchin87} showed that the fibre of the $\mathrm{Sp}$-Hitchin system above a general point in $A_{\mathrm{Sp}}$ corresponding to $C\rightarrow D$ can be identified with the Prym variety of this double cover, namely
$$\mathrm{Prym}(C/D):=\{L\;|\;\sigma^*L\cong L^{\vee}\}^0\subset\mathrm{Jac}^0C,$$
where $\sigma:C\rightarrow C$ is the covering involution of $C\rightarrow D$ and the superscript $^0$ denotes the connected component of the trivial line bundle.

\begin{remark}
Classically, ``Prym variety'' is used to refer to the case where $C\rightarrow D$ is unbranched or has just two branch points, in which case $\mathrm{Prym}(C/D)$ is principally polarized, though here we use ``Prym variety'' in a more general sense, allowing $4n(g-1)>2$ branch points. In particular, $\mathrm{Prym}(C/D)$ will have polarization of type $(1,\ldots,1,2,\ldots,2)$ with $g(D)$ $2$s. Its dimension is given by
$$g(C)-g(D)=(2n^2+n)(g-1)=\mathrm{dim}A_{\mathrm{Sp}}=\frac{1}{2}\mathrm{dim}\mathcal{M}_{\mathrm{Sp}}(2n).$$
\end{remark}

To recover the Higgs bundle $(E,\phi)$ from the spectral data $C\rightarrow D$ and $L\in\mathrm{Prym}(C/D)$ we need to choose a square root of $K_{\Sigma}$, and then $E=\pi_*(L\otimes\pi^*K_{\Sigma}^{n-1/2})$. Note that the projection $\pi:C\rightarrow \Sigma$ is a $2n$-to-$1$ cover so $L\otimes\pi^*K_{\Sigma}^{n-1/2}$ will be a line bundle on $C$ of degree $2n(2n-1)(g-1)$, and we will recover $\mathrm{deg}E=0$ as required. See Hitchin~\cite{hitchin87} or Schaposnik~\cite{schaposnik18} for details.

\subsection{Compact Prym fibrations}

Compact Lagrangian fibrations by Prym varieties were first constructed by Markushevich and Tikhomirov~\cite{mt07} from K3 double covers of degree two del Pezzo surfaces. Other examples were constructed by Arbarello, Sacc{\`a}, and Ferretti~\cite{asf15} from K3 double covers of Enriques surfaces, and by Matteini~\cite{matteini16} from K3 double covers of degree three del Pezzo surfaces. Matteini~\cite{matteini14} considered the general construction from a K3 surface with an anti-symplectic involution. Our description below follows the ideas from these articles.

For the $\mathrm{Sp}$-Hitchin system the spectral curve $C\subset K_{\Sigma}$ is a branched double cover of $D\subset K_{\Sigma}^2$, and the covering involution $C\rightarrow C$ comes from the covering involution $\sigma:K_{\Sigma}\rightarrow K_{\Sigma}$ of $K_{\Sigma}\rightarrow K^2_{\Sigma}$, which is given by multiplication by $-1$ in the fibres and is anti-symplectic. For a compact analogue we start with a branched double cover $\pi:S\rightarrow T$ of a K3 surface over a del Pezzo surface, with anti-symplectic covering involution $\sigma:S\rightarrow S$. Denote by $d=K_T^2\in\{1,\ldots,9\}$ the degree of the del Pezzo surface. The branch locus $\Delta_T\subset T$ must belong to the linear system $|-2K_T|$. Write $\Delta_S\subset S$ for the ramification locus $\pi^{-1}\Delta_T$. For the $\mathrm{Sp}$-Hitchin system the double cover $K_{\Sigma}\rightarrow K_{\Sigma}^2$ is ramified/branched along the zero section, the spectral curves in $K_{\Sigma}$ lie in $|2n\Sigma|$, and their images in $K_{\Sigma}^2$ lie in $|n\Sigma|$ (where we use $\Sigma$ to denote the zero section in both $K_{\Sigma}$ and $K_{\Sigma}^2$). In our compact analogue, we therefore consider the linear systems $|2n\Delta_S|$ and $|n\Delta_T|$ in $S$ and $T$, respectively. In particular, the $\sigma$-fixed locus $|2n\Delta_S|^{\sigma}$ contains a component isomorphic to $|n\Delta_T|$, parametrizing double covers $C\rightarrow D$ where $C\subset S$ is a $\sigma$-invariant curve and $D=C/\sigma$. A calculation shows that $\Delta_S\cong\Delta_T$ has genus $g_{\Delta}=d+1$, $C$ has genus $g_C=4n^2d+1$, $D$ has genus $g_D=n(2n-1)d+1$, and generically the double cover $C\rightarrow D$ has exactly $4nd$ branch points.

We want to construct the relative Prym variety of this family of double covers of curves. For smooth $C\rightarrow D$ the Prym variety $\mathrm{Prym}(C/D)$ is defined as (the component containing $\O_C$ of) the fixed locus of the involution $L\mapsto (\sigma^*L)^{\vee}$ on $\mathrm{Jac}^0C$. We can define the relative Prym variety as the fixed locus of an involution on a family of Jacobians. In fact, we start with the moduli space $\mathcal{M}_S(v)$ of semistable sheaves on $S$ with Mukai vector
$$v:=(0,[2n\Delta_S],-(2n)^2(g_{\Delta}-1))=(0,[2n\Delta_S],-4n^2d)\in\H^{\bullet}(S,\Z),$$
whose general element looks like $\iota_*L$ where $L$ is a degree $0$ line bundle on a smooth curve $C$ in the linear system $|2n\Delta_S|$. This is a singular Lagrangian fibration over $|2n\Delta_S|$: because the Mukai vector is not primitive there will be strictly semistable sheaves where the moduli space is singular. If we choose a $\sigma$-invariant divisor like $\Delta_S$ to polarize $S$, then the anti-symplectic involution $\sigma^*$ will be well-defined on $\mathcal{M}_S(v)$; in particular, it will take semistable sheaves to semistable sheaves.

For a line bundle $L$ on a smooth curve $C$, duality $^{\vee}$ on the fibre $\mathrm{Jac}^0C$ is given by $L^{\vee}:=\mathcal{H}om_C(L,\O_C)$. To extend this to an involution on the whole moduli space we first observe that
$$\iota_*L^{\vee}:=\iota_*\mathcal{H}om_C(L,\O_C)\cong\mathcal{H}om_S(\iota_*L,\O_C)\cong\mathcal{E}xt_S^1(\iota_*L,\O_S(-C)).$$
where $\iota:C\hookrightarrow S$ is inclusion. To see the last isomorphism, apply $\mathcal{H}om_S(\iota_*L,-)$ to the short exact sequence
$$0\longrightarrow\O_S(-C)\longrightarrow\O_S\longrightarrow\O_C\longrightarrow 0,$$
giving
$$0\longrightarrow\mathcal{H}om_S(\iota_*L,\O_C)\longrightarrow\mathcal{E}xt_S^1(\iota_*L,\O_S(-C))\longrightarrow\mathcal{E}xt_S^1(\iota_*L,\O_S).$$
The last map is induced by multiplication by the section of $\O_S(C)$ that vanishes on $C$, and is therefore the zero map.

Now $\mathcal{E}xt_S^1(-,\O_S(-C))$ is better behaved on the whole moduli space. Indeed, for a sheaf $\mathcal{F}$ of pure dimension one on $S$, $\mathcal{E}xt_S^i(\mathcal{F},\O_S(-C))$ vanishes for $i\neq 1$ by Proposition~1.1.10 of Huybrechts and Lehn~\cite{hl97}. Up to tensoring with a line bundle
$$\mathcal{E}xt_S^1(\mathcal{F},\O_S(-C))\cong\mathcal{E}xt_S^1(\mathcal{F},\omega_S)\otimes\O_S(-C))$$
is the dual sheaf of $\mathcal{F}$, according to Definition~1.1.7 of~\cite{hl97}, and the dual of the dual is isomorphic to the original sheaf by Proposition~1.1.10 of~\cite{hl97}. Moreover, one can show that $\mathcal{F}\mapsto\mathcal{E}xt_S^1(\mathcal{F},\O_S(-C))$ preserves semistability and is therefore an involution on $\mathcal{M}_S(v)$. This involution is antisymplectic because it is antisymplectic on the open set corresponding to line bundles supported on smooth curves, where it is given in local coordinates by multiplication by $-1$ in the fibres of the Lagrangian fibration $\mathcal{M}_S(v)\rightarrow |2n\Delta_S|$.

These involutions commute and their composition gives a symplectic involution
\begin{eqnarray*}
\tau:\mathcal{M}_S(v) & \longrightarrow & \mathcal{M}_S(v) \\
\mathcal{F} & \longmapsto & \mathcal{E}xt_S^1(\sigma^*\mathcal{F},\O_S(-C)).
\end{eqnarray*}
Finally, we define our relative Prym variety to be
$$\mathcal{P}_{S/T}(v):=\mathrm{Fix}(\tau)^0\subset\mathcal{M}_S(v)$$
where the superscript $^0$ denotes the connected component of $\iota_*\O_C$.

Note that a sheaf $\mathcal{F}$ in $\mathcal{M}_S(v)$ is supported on a curve $C$ in $|2n\Delta_S|$, and if $\mathcal{F}$ is fixed by $\tau$ then $C$ must be $\sigma$-invariant. Assume in addition that $C$ is smooth; if $\mathcal{F}$ is fixed by $\tau$ then it must lie in $\mathrm{Prym}(C/D)$. In this way, taking the fixed locus of $\tau$ picks out both $\sigma$-invariant curves and line bundles in the Prym varieties of those curves.

\begin{thm}[\cite{mt07,matteini14,asf15,matteini16}]
The relative Prym variety $\mathcal{P}_{S/T}(v)$ is a singular Lagrangian fibration over $|n\Delta_T|\cong\P^{n(2n+1)d}$. Its general fibre is a Prym variety $\mathrm{Prym}(C/D)$ for a smooth double cover of curves $C\rightarrow D$, which is an abelian variety of dimension
$$g_C-g_D=n(2n+1)d.$$
\end{thm}

\begin{definition}
For convenience, in this article we will call $\mathcal{P}_{S/T}(v)\rightarrow\P^{n(2n+1)d}$ the K3-del Pezzo system.
\end{definition}

\begin{remark}
Strictly speaking, these systems do not include those studied by Markushevich and Tikhomirov~\cite{mt07} and by Matteini~\cite{matteini16}: in their examples $C$ lies in a linear system different to $|2n\Delta_S|$. Likewise, the examples of Arbarello, Sacc{\`a}, and Ferretti~\cite{asf15} come from K3 double covers of Enriques surfaces. However, the construction follows all of their ideas so we cite them in the theorem above.
\end{remark}

\begin{remark}
A major goal in this area is to construct new examples of smooth holomorphic symplectic manifolds. Unfortunately the examples above (and in~\cite{mt07,matteini14,matteini16}) always have singularities that do not admit symplectic resolutions. Something similar is true if we start with a K3 double cover of an Enriques surface: in some cases a symplectic resolution does not exist, while in other cases it exists but leads to a known example~\cite{asf15}. The problem arises when $C=C_1\cup C_2$ is the union of two $\sigma$-invariant curves that meet in $C_1.C_2\geq 4$ points. In this case a local analysis of $\mathcal{P}_{S/T}(v)$ shows that it has $\mathbb{Q}$-factorial and terminal singularities at points corresponding to polystable sheaves supported on $C$; it follows that no symplectic resolution exists (see Sections~4 and~5 of~\cite{asf15} for details).

\end{remark}

\subsection{The deformation}

In this section we describe how to deform the K3-del Pezzo system to a compactification of the $\mathrm{Sp}$-Hitchin system. We first describe the deformation of surface double covers.

\subsubsection{Deforming surface double covers}

As described in Section~\ref{K3curve}, if a curve $\Sigma$ (that we assume to be non-hyperelliptic) is contained in a K3 surface $S$ then we can deform $S$ to the cone over the canonical embedding of $\Sigma$. Specifically, there is a family of surfaces $X\rightarrow\P^1$ such that $X_0\cong \overline{K}_{\Sigma}$ and $X_t\cong S$ for all $t\neq 0$. Now let us assume that the K3 surface is a double cover $\pi:S\rightarrow T$ over a degree $d$ del Pezzo surface, branched over $\Delta_T$, and let the curve $\Sigma$ be $\Delta_S=\pi^{-1}\Delta_T$. There are two ways to proceed: either we extend the involution on $S$ to an involution on $X$, or we describe a deformation of $T$ to $\overline{K^2_{\Sigma}}$ and then define its double cover.

For the first approach we use the description of $X$ from the remark in Section~\ref{K3curve}: namely, $X$ is $S\times\P^1$ with $\Sigma\subset S\times\{0\}$ blown up and the proper transform of $S\times\{0\}$ blown down. The involution on $S$ induces an involution on $S\times\P^1$ in the obvious way. Because this involution fixes $\Sigma\subset S\times\{0\}$, the involution extends to the blow-up
$$\mathrm{Blow}_{\Sigma}S\times\P^1.$$
Moreover, the induced action on the normal bundle $K_{\Sigma}\oplus\O_{\Sigma}$ of $\Sigma$ in $S\times\P^1$ is by $-1$ on $K_{\Sigma}$ and $1$ on $\O_{\Sigma}$, so the induced action on the exceptional locus $\P(K_{\Sigma}\oplus\O_{\Sigma})$ is by $-1$ on the $\P^1$-fibres, fixing the zero and infinity sections. The proper transform of $S\times\{0\}$ is preserved by the action and so there is an induced action on the blow-down, which is $X$. The final result is an involution on $X$, that we again denote by $\sigma$, which is
\begin{itemize}
\item multiplication by $-1$ in the fibres of $X_0\cong\overline{K}_{\Sigma}$, fixing the zero section $\Sigma\subset K_{\Sigma}$ and the singular point at infinity of $\overline{K}_{\Sigma}$,
\item the covering involution of $S\rightarrow T$ for $X_t\cong S$ when $t\neq 0$.
\end{itemize}
Quotienting by this involution gives a family of surfaces $Y\rightarrow\P^1$ such that $Y_0\cong\overline{K^2_{\Sigma}}$ and $Y_t\cong T$ for $t\neq 0$.

\begin{remark}
The fixed locus of the involution on the blow-up $\mathrm{Blow}_{\Sigma}S\times\P^1$ consists of two components: the infinity section of the exceptional locus $\P(K_{\Sigma}\oplus\O_{\Sigma})$ and the proper transform of $\Sigma\times\P^1$, which meets the exceptional locus $\P(K_{\Sigma}\oplus\O_{\Sigma})$ along the zero section. Because the first component is codimension two, if we quotient $\mathrm{Blow}_{\Sigma}S\times\P^1$ by the involution we will get a family of $A_1$-singularities along this component: locally it will look like $\C\times(\C^2/\pm 1)$. We will need this observation shortly.
\end{remark}

For the second approach we start with the threefold $T\times\P^1$ and blow up the curve
$$\Sigma\cong\Delta_T\subset T\times\{0\}\subset T\times\P^1.$$
The result, $\mathrm{Blow}_{\Sigma}T\times\P^1$, should be the quotient of $\mathrm{Blow}_{\Sigma}S\times\P^1$ by the involution. However, as observed above, this quotient is singular, whereas $\mathrm{Blow}_{\Sigma}T\times\P^1$ is smooth. To correct this, we instead perform a weighted blow-up, as described in Section 10 of Koll{\'a}r and Mori~\cite{km92} using toric geometry.

\begin{example}
The cone spanned by $e_1$ and $e_2$ inside $\mathbb{R}^2$ corresponds to the affine toric variety $\C^2$. If $a_1$ and $a_2$ are coprime positive integers, then introducing the vector $a_1e_1+a_2e_2$ corresponds to a weighted blow-up $\mathrm{Blow}_{a_1,a_2}\C^2$ of the origin $(0,0)$ in $\C^2$, with exceptional locus isomorphic to the weighted projective line $\P(a_1,a_2)$. Moreover, $\mathrm{Blow}_{a_1,a_2}\C^2$ can be covered by two charts, isomorphic to $\C^2/\Z_{a_1}$ and $\C^2/\Z_{a_2}$, where the cyclic groups act by
$$(x,y)\sim (e^{2\pi i/a_1} x,e^{-a_22\pi i/a_1}y)\qquad\mbox{and}\qquad (x,y)\sim (e^{-a_12\pi i/a_2} x,e^{2\pi i/a_2} y)$$
respectively. (This corresponds to $m=1$ in the notation of~\cite{km92}.)

In particular, for $(a_1,a_2)=(1,2)$ we get a weighted blow-up $\mathrm{Blow}_{1,2}\C^2$ of $\C^2$ which is covered by two charts $\C^2$ and $\C^2/\pm 1$. Note that $0$ in the exceptional locus $\P(1,2)$ is a smooth point of $\mathrm{Blow}_{1,2}\C^2$, whereas $\infty\in\P(1,2)$ is the $A_1$-singularity $\C^2/\pm 1$.
\end{example}

We can perform a weighted blow-up of the curve
$$\Sigma\cong\Delta_T\subset T\times\{0\}\subset T\times\P^1.$$
Locally this will look like $\C\times\mathrm{Blow}_{1,2}\C^2$, where the first $\C$ is in the direction along the curve $\Sigma$. Now because $\Delta_T$ is in the linear system $|-2K_T|$, its normal bundle in $T$ is $K^{-2}_T|_{\Delta_T}$, which by adjunction is $K_{\Delta_T}^2$. The normal bundle of $\Sigma$ in $T\times\P^1$ is therefore $K^2_{\Sigma}\oplus\O_{\Sigma}$ and the exceptional locus of the weighted blow-up is $\P_{1,2}(K^2_{\Sigma}\oplus\O_{\Sigma})$. The total space will be smooth along the zero section of $\P_{1,2}(K^2_{\Sigma}\oplus\O_{\Sigma})$, whereas it will have a family of $A_1$-singularities along the infinity section. This agrees with the quotient of $\mathrm{Blow}_{\Sigma}S\times\P^1$ by the involution. Finally, blowing down the proper transform of $T\times\{0\}$ gives $Y$, a trivial fibration over $\P^1-\{0\}$ with fibres isomorphic to $T$, and fibre over $0$ given by $\P_{1,2}(K^2_{\Sigma}\oplus\O_{\Sigma})$ with the section at infinity blown down, i.e., the one-point compactification $\overline{K^2_{\Sigma}}$ of $K^2_{\Sigma}$.

\begin{remark}
If $d>1$ then $\O(\Delta_T)=\O(-2K_T)$ is very ample on $T$ and we have an embedding
$$T\hookrightarrow\P(\H^0(T,\Delta_T)^{\vee})\cong \P^{3d}$$
which restricted to $\Delta_T\subset T$ gives the bicanonical embedding $\Delta_T\hookrightarrow\P^{3d-1}$ (recall that $\Delta_T$ has genus $g_{\Delta}=d+1$). Take the cone over $T$ in $\P^{3d+1}$, take the pencil of hyperplanes in $\P^{3d+1}$ that contain $\Delta_T$, and intersect them with the cone over $T$ to get a pencil of surfaces; this is precisely the family of surfaces $Y\rightarrow\P^1$.
\end{remark}

To complete our description in the second approach we need to define the double cover of $Y$. We do this by describing the branch locus. We start with the divisor
$$\Delta_T\times\P^1\subset T\times\P^1.$$
After the weighted blow up of the curve
$$\Sigma\cong\Delta_T\subset T\times\{0\}\subset T\times\P^1,$$
we take the proper transform of $\Delta_T\times\P^1$ plus the singular locus, i.e., the infinity section of the exceptional locus $\P_{1,2}(K^2_{\Sigma}\oplus\O_{\Sigma})$. Finally we blow down the proper transform of $T\times\{0\}$ to get $Y$. The resulting branch locus is the proper transform of $\Delta_T\times\P^1$ union the point at infinity in $Y_0\cong\overline{K^2_{\Sigma}}$.

It is clear that these two approaches give the same result: the double cover of $Y$ branched over the given branch locus is $X$, and $X$ quotiented by the involution $\sigma$ is $Y$. In summary we have a $\P^1$-family of branched double covers of surfaces
$$\begin{array}{ccc}
 X & & \\
 \downarrow & \searrow & \\
 Y & \longrightarrow & \mathbb{P}^1
\end{array}$$
with $X_t\rightarrow Y_t$ isomorphic to $S\rightarrow T$ for $t\neq 0$ and $X_0\rightarrow Y_0$ isomorphic to $\overline{K}_{\Sigma}\rightarrow\overline{K^2_{\Sigma}}$.

\subsubsection{Deforming integrable systems}

To construct a $\P^1$-family of Prym fibrations from $X\rightarrow Y$ we imitate the ideas for the construction of the K3-del Pezzo system. Namely, we define a moduli space of sheaves on $X$ and then take the fixed locus of a (fibrewise symplectic) involution on this moduli space.

The ramification curve of $\overline{K}_{\Sigma}\rightarrow\overline{K^2_{\Sigma}}$ is the zero section $\Sigma\subset\overline{K}_{\Sigma}$. It deforms to the ramification curve $\Delta_S\subset S$ of $S\rightarrow T$. Let $w\in\H^{\bullet}(X,\Z)$ denote the Chern class of a degree $0$ line bundle on a curve in one of the linear systems $|2n\Delta_S|$, pushed forward to give a torsion sheaf on $X$. Define $\mathcal{M}_{X/\P^1}(w)$ to be the moduli space of semistable sheaves on $X$ with Chern class $w$ that are supported in fibres of $X\rightarrow\P^1$. Each such sheaf is supported on a curve in a linear system $|2n\Delta_S|$ for some $X_t$, or in $|2n\Sigma|$ for $X_0=\overline{K}_{\Sigma}$. These linear systems form a $\P^{4n^2d+1}$-bundle $P$ over $\P^1$, and there is a scheme-theoretic support map
$$\begin{array}{ccc}
\mathcal{M}_{X/\P^1}(w) & \longrightarrow & P \\
 & \searrow & \downarrow \\
 & &\mathbb{P}^1.
\end{array}$$

The involution $\sigma$ on $X$ induces an involution $\sigma^*$ on sheaves on $X$. Assume that $X$ is polarized by a $\sigma$-invariant divisor (in fact, a $\sigma$-invariant relatively ample divisor on $X\rightarrow\P^1$ will suffice; the proper transform of $\Delta_S\times\P^1\subset S\times\P^1$ is an example). Then $\sigma^*$ will preserve semistability and give an involution of $\mathcal{M}_{X/\P^1}(w)$, anti-symplectic on the fibres of $\mathcal{M}_{X/\P^1}(w)\rightarrow\P^1$.

Next we want to extend duality on a fibre to an involution on the whole moduli space. Actually what we describe below is still only a rational involution, but it will suffice for our construction.

\begin{lemma}
Let $\iota_*L\in\mathcal{M}_{X/\P^1}(w)$, where $L$ is a line bundle on a smooth curve $\iota:C\hookrightarrow X$ contained in a fibre $X_t$ with $t\neq 0$. Then
$$\iota_*L^{\vee}:=\iota_*\mathcal{H}om_C(L,\O_C)\cong\mathcal{H}om_X(\iota_*L,\O_C)\cong\mathcal{E}xt_X^2(\iota_*L,\O_X(-X_t-2n\Delta))$$
where $\Delta$ is the proper transform of $\Delta_S\times\P^1\subset S\times\P^1$ in $X$.
\end{lemma}

\begin{proof}
We have already seen that applying $\mathcal{H}om_{X_t}(\iota_*L,-)$ to the short exact sequence
$$0\longrightarrow\O_{X_t}(-C)\longrightarrow\O_{X_t}\longrightarrow\O_C\longrightarrow 0$$
gives
$$\mathcal{H}om_{X_t}(\iota_*L,\O_C)\cong\mathcal{E}xt_{X_t}^1(\iota_*L,\O_{X_t}(-C))=\mathcal{E}xt_{X_t}^1(\iota_*L,\O_{X_t}(-2n\Delta\cap X_t)).$$
(Here we are temporarily using $\iota$ to denote the inclusion $C\hookrightarrow X_t$.)
Similarly, applying $\mathcal{H}om_X(\iota_*L,-)$ to the short exact sequence
$$0\longrightarrow\O_X(-X_t-2n\Delta)\longrightarrow\O_X(-2n\Delta)\longrightarrow\O_{X_t}(-2n\Delta)\longrightarrow 0$$
gives
$$\mathcal{E}xt^1_X(\iota_*L,\O_X(-2n\Delta))\longrightarrow\mathcal{E}xt^1_X(\iota_*L,\O_{X_t}(-2n\Delta))\longrightarrow\hspace*{40mm}$$
$$\hspace*{40mm}\longrightarrow\mathcal{E}xt_X^2(\iota_*L,\O_X(-X_t-2n\Delta))\longrightarrow\mathcal{E}xt_X^2(\iota_*L,\O_X(-2n\Delta)).$$
The first term vanishes by Proposition~1.1.6 of Huybrechts and Lehn~\cite{hl97} because $\iota_*L$ is of dimension one, and although $X$ is not smooth, the support of $\iota_*L$ lies in the smooth locus of $X$. The last map is induced by multiplication by the section of $\O_X(X_t)$ that vanishes on $X_t$, and is therefore the zero map. Therefore we have
$$\mathcal{E}xt^1_X(\iota_*L,\O_{X_t}(-2n\Delta))\cong\mathcal{E}xt_X^2(\iota_*L,\O_X(-X_t-2n\Delta)),$$
and combining this with the earlier isomorphism (after pushing everything forward to $X$) completes the proof.
\end{proof}

\begin{remark}
The lemma is also true if the smooth curve $C$ is contained in the special fibre $X_0$, provided it does not meet the apex of the cone $X_0\cong\overline{K}_{\Sigma}$. We only require that the support $C$ of $\iota_*L$ lie in the smooth locus of $X$.
\end{remark}

\begin{lemma}
Define a map $\mathcal{M}_{X/\P^1}(w)\dashrightarrow\mathcal{M}_{X/\P^1}(w)$ by
$$\mathcal{F}\longmapsto\mathcal{E}xt_X^2(\mathcal{F},\O_X(-X_t-2n\Delta)).$$
Then this map is well-defined on the locus of sheaves $\mathcal{F}\in\mathcal{M}_{X/\P^1}(w)$ whose support is contained in $X_t$ for $t\neq 0$ or is contained in $X_0$ but does not meet the apex of the cone $X_0\cong\overline{K}_{\Sigma}$. Moreover, it is a rational involution on $\mathcal{M}_{X/\P^1}(w)$ and anti-symplectic on the fibres of $\mathcal{M}_{X/\P^1}(w)\rightarrow\P^1$.
\end{lemma}

\begin{proof}
As observed by Coughlan and Sano~\cite{cs18}, the cone over a K3 surface is a normal, Gorenstein variety with an isolated log canonical singularity. The same is true for $X$, because it is just the cone over $S$ with the smooth curve $\Sigma=\Delta_S\subset S$ blown up. It follows that results like Serre duality and Kodaira vanishing hold on $X$ (for example, see Fujino~\cite{fujino09}). Therefore the proof of Proposition~1.1.6 of Huybrechts and Lehn~\cite{hl97} works in this setting to show that for a sheaf $\mathcal{F}$ of dimension one on $X$, $\mathcal{E}xt_X^i(\mathcal{F},\O_X(-X_t-2n\Delta))$ vanishes for $i=0$ and $1$. Up to tensoring with a line bundle
$$\mathcal{E}xt_X^2(\mathcal{F},\O_X(-X_t-2n\Delta))\cong\mathcal{E}xt_X^2(\mathcal{F},\omega_X)\otimes\omega_X^{-1}(-X_t-2n\Delta))$$
is the dual sheaf of $\mathcal{F}$, according to Definition~1.1.7 of~\cite{hl97}. Moreover, if $\mathcal{F}$ is a pure sheaf of dimension one whose support lies in the smooth locus of $X$ then the dual of its dual is isomorphic to $\mathcal{F}$ by Proposition~1.1.10 of~\cite{hl97}. One can also show that
$$\mathcal{F}\longmapsto\mathcal{E}xt_X^2(\mathcal{F},\O_X(-X_t-2n\Delta))$$
preserves semistability, and it therefore gives a rational involution on $\mathcal{M}_{X/\P^1}(w)$, defined for all sheaves $\mathcal{F}$ except those whose support is contained in $X_0$ and meets the apex of the cone $X_0\cong\overline{K}_{\Sigma}$. It is anti-symplectic on the fibres because the previous lemma shows that it is anti-symplectic on the open set where $\mathcal{F}=\iota_*L$ for a line bundle on a smooth curve $C\subset X_t$, $t\neq 0$, where it is given in local coordinates by multiplication by $-1$ in the fibres of the Lagrangian fibration $\mathcal{M}_{X_t}(w)\rightarrow P_t$.
\end{proof}

\begin{remark}
If the support of $\mathcal{F}$ is contained in $X_0$ and meets the apex of the cone $X_0\cong\overline{K}_{\Sigma}$ then $\mathcal{E}xt_X^2(\mathcal{F},\O_X(-X_t-2n\Delta))$ is still a sheaf of dimension one, but it is possible that $\mathcal{E}xt_X^3(\mathcal{F},\O_X(-X_t-2n\Delta))$, which is supported at the apex of $X_0$, is also nonzero. This would mean that $\mathcal{E}xt_X^2(\mathcal{F},\O_X(-X_t-2n\Delta))$ would not have the required Mukai vector $w$, for instance.
\end{remark}

\begin{question}
Is this involution well-defined on all of $\mathcal{M}_{X/\P^1}(w)$? If $\mathcal{F}$ is isomorphic to the pushforward $\iota_*\mathcal{E}$ of a locally free sheaf $\mathcal{E}$ on its support $\iota:C\hookrightarrow X$ then we expect that $\mathcal{E}xt_X^2(\mathcal{F},\O_X(-X_t-2n\Delta))$ will be isomorphic to $\iota_*\mathcal{E}^{\vee}$ and $\mathcal{E}xt_X^3(\mathcal{F},\O_X(-X_t-2n\Delta))$ will vanish, even if the support is contained in $X_0$ and meets the apex of the cone $X_0\cong\overline{K}_{\Sigma}$. This suggests that the indeterminacy of the involution has codimension at least three, although one must be careful as the complement of $\mathrm{Jac}^0C$ in $\overline{\mathrm{Jac}}^0C$ might not be a divisor for a highly singular curve $C$ (in general, $\overline{\mathrm{Jac}}^0C$ could consist of several irreducible components, including components of dimension greater than the genus of $C$). Moreover, it is not clear that the indeterminacy having large codimension would necessarily imply that the involution can be extended to all of $\mathcal{M}_{X/\P^1}(w)$.
\end{question}

In any case, the rational involution will commute with the involution $\sigma^*$, so we can compose them to get a rational involution
\begin{eqnarray*}
\tau:\mathcal{M}_{X/\P^1}(w) & \dashrightarrow & \mathcal{M}_{X/\P^1}(w) \\
\mathcal{F} & \longmapsto & \mathcal{E}xt_X^2(\sigma^*\mathcal{F},\O_X(-X_t-2n\Delta)).
\end{eqnarray*}
The following is inspired by Definition~3.9 of Arbarello, Sacc{\`a}, and Ferretti~\cite{asf15}.

\begin{definition}
Let $\mathrm{Fix}(\tau)$ be the fixed locus of $\tau$ on the open subset of $\mathcal{M}_{X/\P^1}(w)$ where $\tau$ is regular, let $\mathrm{Fix}(\tau)^0$ be the component containing $\iota_*\O_C$ for some smooth curve $\iota:C\hookrightarrow X_t$ in the linear system $|2n\Delta_S|$, $t\neq 0$, and define
$$\mathcal{P}_{X/Y}(w):=\overline{\mathrm{Fix}(\tau)^0}\subset\mathcal{M}_{X/\P^1}(w)$$
to be the closure.
\end{definition}

The composition of two anti-symplectic maps is symplectic. Therefore $\tau$ preserves the fibrewise symplectic structure on $\mathcal{M}_{X/\P^1}(w)$ and its fixed locus $\mathcal{P}_{X/Y}(w)$ will be a $\P^1$-family of symplectic varieties. The following theorem is now clear.

\begin{thm}
For $t\neq 0$ we have $X_t/Y_t\cong S/T$ and $\mathcal{P}_{X/Y}(w)_t$ can be identified with the K3-del Pezzo system
$$\mathcal{P}_{S/T}(v)\rightarrow |n\Delta_T|\cong\P^{n(2n+1)d}.$$
For $t=0$ we have $X_0/Y_0\cong \overline{K}_{\Sigma}/\overline{K^2_{\Sigma}}$ and $\mathcal{P}_{X/Y}(w)_0$ can be identified with the natural compactification of the $\mathrm{Sp}$-Hitchin system $\mathcal{M}_{\mathrm{Sp}}(2n)\rightarrow A_{\mathrm{Sp}}$ induced by the inclusions $K_{\Sigma}\subset \overline{K}_{\Sigma}$, $K^2_{\Sigma}\subset\overline{K^2_{\Sigma}}$, and $A_{\mathrm{Sp}}\subset |n\Sigma|$. In this way we get a deformation of the K3-del Pezzo system to a compactification of the $\mathrm{Sp}$-Hitchin system.
\end{thm}

\begin{flushleft}
Department of Mathematics\hfill sawon@email.unc.edu\\
University of North Carolina\hfill cshen@live.unc.edu\\
Chapel Hill NC 27599-3250\hfill sawon.web.unc.edu\\
USA\\
\end{flushleft}

\end{document}